\documentclass[reqno]{amsart}

\usepackage{amssymb}
\usepackage{mathrsfs}
\usepackage{cite}
\usepackage{etoolbox}
\usepackage[all,cmtip]{xy}

\usepackage{enumitem}

\usepackage{graphicx}

\usepackage[T1]{fontenc}
\usepackage[hidelinks]{hyperref}
\makeatletter
\patchcmd{\abstract}
  {\leftmargin3pc \rightmargin\leftmargin}
  {\leftmargin1pc \rightmargin\leftmargin}
  {}{}
\makeatother

\newtheorem{theorem}{Theorem}[section]

\newtheorem{lemma}[theorem]{Lemma}

\theoremstyle{definition}

\numberwithin{equation}{section}

\allowdisplaybreaks

\newcommand{\R}{\mathbb R}

\newcommand{\T}{\mathbb T}

\newcommand{\cal}{\mathcal}

\begin{document}

\title[A note on the partial sum of Dirichlet series]{A note on the partial sum of bounded Dirichlet series}

\author{Yukun Chen}
\address{School of Mathematical Science, Fudan University, Shanghai 200433, China}
\email{ykchen26@m.fudan.edu.cn}

\author{Xiangdi Fu}
\address{School of Fundamental Physics and Mathematical Sciences, HIAS, University of Chinese Academy of Sciences, Hangzhou 310024, China}
\email{xdfu@ucas.ac.cn}

\date{}

\begin{abstract}
Let $\mathcal H^\infty$ be the space of all Dirichlet series that admit a bounded holomorphic extension to the open right half-plane
$
\{s\in \mathbb C: \operatorname{Re} s >0\},
$ and let $$
\mathcal S_N: \mathcal H^\infty \to \mathcal H^\infty;
\sum_{n=1}^\infty a_n n^{-s} \mapsto \sum_{n=1}^N a_n n^{-s}.
$$ be the $N$-th partial sum operator. This note establishes the asymptotic lower bound
$$
\liminf_{N\to \infty} \frac{\|\mathcal S_N\|_{\mathcal H^\infty \to \mathcal H^\infty}}{\log N} \geq \frac{1}{2\pi}.
$$
Together with the upper bound of R. Balasubramanian, B. Calado, and H. Queff\'{e}lec, this shows that the growth of $\|\mathcal S_N\|_{\mathcal H^\infty\to \mathcal H^\infty}$ is of sharp logarithmic order.
\end{abstract}



\maketitle

\section{Introduction}

For a formal power series $f=\sum_{n=0}^\infty a_n z^n$, we denote its $N$-th partial sum by
$$
S_N f:=\sum_{n=0}^N a_nz^n.
$$ The partial sum operators $S_N$ are among the most basic operators arising in the study
of various analytic function spaces. In particular, partial sum operators on $H^p$, the classical Hardy spaces over the unit disc, have been extensively studied. It is well known that, for $1<p<\infty$, the partial sum operators are uniformly bounded on $H^p$ and the exact supremum is given by Hollenbeck and Verbitsky \cite{HV}: 
$$
\sup_{N\geq 1}
\|S_N\|_{H^p\to H^p}= \csc(\pi/p).
$$
For $p=1$, the partial sum operators are no longer uniformly bounded, but it is also known that 
$$
\|S_N\|_{H^1 \to H^1}
\asymp \log N.
$$
At the endpoint $p=\infty$, an elegant theorem of Landau \cite{Landau} gives the exact value
$$
\|S_N\|_{H^\infty \to H^\infty}
=
\sum_{k=0}^N
\frac{1}{16^k} \binom{2k}{k}^2=
\frac{1}{\pi}\log N+O(1).
$$

In recent years, the study of Hardy spaces $\cal H^p$ of Dirichlet series has also received some attention \cite{HLS,Ba,DGMS,QQ}. For finite $p$, the space $\mathcal H^p$ is defined as the completion of Dirichlet polynomials with respect to the $p$-th Besicovitch norm. For $\sigma\in\mathbb R$, let $\mathbb C_\sigma$ denote the right half-plane
$$
\mathbb C_\sigma:=\{s\in\mathbb C:\operatorname{Re}s>\sigma\},
$$
and let $H^\infty(\mathbb C_\sigma)$ denote the space of bounded holomorphic functions on $\mathbb C_\sigma$. The space $\mathcal H^\infty$ consists of all functions $f\in H^\infty(\mathbb C_0)$ that admit a Dirichlet series representation converging for sufficiently large $\operatorname{Re}s$. The norm on $\mathcal H^\infty$ is defined by
$$
\|f\|_{\mathcal H^\infty}:=\sup_{s\in\mathbb C_0}|f(s)|.
$$ In other words, a function $f$ lies in $\mathcal H^\infty$ if it satisfies the following two conditions:
\begin{itemize}
\item $f$ is a bounded holomorphic function on $\mathbb C_0$;
\item there exist $\sigma>0$ and coefficients $(a_n)_{n\geq1}$ such that the Dirichlet series
$$
\sum_{n=1}^\infty a_n n^{-s}
$$
converges on $\mathbb C_{\sigma}$ and represents $f$ there.
\end{itemize}
Indeed, the coefficients $(a_n)_{n\geq1}$ are uniquely determined by $f$, and the space $\cal H^\infty$ is complete \cite[Theorem 6.2.1]{QQ}. Moreover, Bohr's theorem \cite[Theorem 6.2.3]{QQ} asserts that whenever $f\in \cal H^\infty$, the corresponding Dirichlet series converges uniformly on $\mathbb C_\varepsilon$ for every $\varepsilon>0$. In particular, it converges at every point of $\mathbb C_0$ and represents $f$ throughout $\mathbb C_0$. For $N\geq1$, the partial sum operator on $\mathcal H^p$ is defined by
$$
\cal S_Nf(s):=\sum_{n=1}^N a_n n^{-s}
$$ for $f(s)=\sum_{n=1}^\infty a_n n^{-s}$.

Partial sum operators on Hardy spaces of Dirichlet series have also been studied. In particular, Konyagin, Queff\'{e}lec, Saksman and Seip \cite{KQSS} showed that, in the case $1<p<\infty$, one has $$\sup_{N\geq 1} \|\mathcal S_N\|_{\mathcal H^p\to \mathcal H^p} = \sup_{N\geq 1} \| S_N\|_{H^p\to H^p}= \csc(\pi/p).$$ At the endpoints, however, the situation is less well understood. For $\mathcal H^1$, Bondarenko, Brevig, Saksman and Seip \cite{BBSS} proved that
\begin{align}\label{eq:p=1}
\|\cal S_N\|_{\mathcal H^1\to\mathcal H^1}
\lesssim
\frac{\log N}{\log\log N}.
\end{align} For $\cal H^\infty$, Balasubramanian, Calado and Queff\'{e}lec \cite{BCQ} established the logarithmic upper bound
\begin{align}\label{eq:p=infty}
\|\cal S_N\|_{\mathcal H^\infty\to\mathcal H^\infty}
\lesssim \log N.
\end{align}
Their proof of \eqref{eq:p=infty} is based on the Perron--Landau formula and a contour integration argument. It is worth noting that the slightly sharper estimate
$$
\|\cal S_N\|_{\cal H^\infty \to \cal H^\infty}
\leq \frac{1}{\pi}\log N+\frac{2}{\pi}\log\log N+O(1)
$$
can be deduced simply by optimizing the parameters in the proof given in \cite{KQSS}.
Moreover, the lower bound of order $\log\log N$ can be easily deduced for both \eqref{eq:p=1} and \eqref{eq:p=infty} by considering Dirichlet series supported on the powers of a single prime.

The purpose of this note is to provide the following new lower bound for \eqref{eq:p=infty}.
\begin{theorem}\label{thm-main}
For the partial sum operator $\cal S_N$ on $\mathcal H^\infty$, we have
$$
\liminf_{N\to\infty}
\frac{\|\cal S_N\|_{\mathcal H^\infty\to\mathcal H^\infty}}
{\log N}
\geq \frac{1}{2\pi}.
$$ In particular, combined with the upper estimate \eqref{eq:p=infty}, this implies
$$
\|\cal S_N\|_{\mathcal H^\infty\to\mathcal H^\infty}
\asymp \log N.
$$
\end{theorem}

\section{Proof}

We begin by recalling some standard facts. Let $(p_k)_{k\geq 1}$ be the sequence of prime numbers, and let $\mathfrak X$ denote the group of all completely multiplicative characters
$$
\chi:\mathbb N_+ \longrightarrow\T,
$$
endowed with the topology of pointwise convergence. By the fundamental theorem of arithmetic, the map
$$
\Theta :\mathfrak X\longrightarrow\T^\infty;
\quad
\chi\longmapsto \big(\chi(p_k)\big)_{k\geq 1},
$$
is a topological group isomorphism. In this sense, we shall identify $\mathfrak X$ with $\T^\infty$ and write $\chi\in\T^\infty$ for a completely multiplicative character.

For each $t\in \mathbb R$, define the character $K_t$ by $$K_t(n):=n^{-it}, \quad n\in \mathbb N_+.$$
It follows from the Kronecker theorem that $\{K_t: t\in \R\}$ is dense in $\T^\infty$; see \cite[Theorem 2.2.4]{QQ} for example. Consequently, for every Dirichlet polynomial $f(s)=\sum_{n=1}^N a_n n^{-s}$, it holds that
\begin{align}\label{eq:Bohr}
\|f\|_{\cal H^\infty}
=\sup_{t\in\R}\left|\sum_{n=1}^N a_nn^{-it}\right|
=\sup_{\chi\in\T^\infty}\left|\sum_{n=1}^N  a_n\chi(n)\right|.
\end{align}

\subsection{The hyperbolic Dirichlet polynomials}
Let $N\geq 10$ be an integer, and let $$d=d(N):=\left\lfloor \sqrt{N} \right\rfloor -1.$$ For a $d\times d$ matrix $A=(a_{ij})_{1\leq i,j \leq d}$, we define the hyperbolic Dirichlet polynomial associated with the coefficient matrix $A$ by $$f_A:= \sum_{1\leq i, j \leq d} a_{ij} \bigg(i \left\lfloor \frac{N}{j} \right\rfloor \bigg)^{-s}.$$ Let \(\ell^p_d\) denote the space \(\mathbb C^d\) endowed with the usual
\(\ell^p\)-norm. We shall identify each \(d\times d\) matrix \(A\) with the
linear operator on \(\mathbb C^d\) represented by \(A\) in the standard basis.
Accordingly, we write $\|A\|_{\ell^p_d\to\ell^q_d}$ for its operator norm as a linear map from \(\ell^p_d\) to \(\ell^q_d\).

The key point is that the uniform norm of a hyperbolic Dirichlet polynomial can be controlled by the \(\ell^\infty\to \ell^1\) norm of the associated matrix \(A\). This is made precise in the following general lemma.

\begin{lemma}\label{lem-matrix-dominated}
    Suppose that $m_i$ and $n_i$, with $1\leq i \leq d$, are positive integers, and that $A=(a_{ij})$ is a $d\times d$ matrix. Let \begin{align}\label{eq-hyperbolic-type-f}
      f=\sum_{1\leq i,j\leq d} a_{ij} (m_in_j)^{-s}.
    \end{align}
    Then $\|f\|_{\cal H^\infty} \leq \|A\|_{\ell^\infty_d \to \ell^1_d}.$
\end{lemma}
\begin{proof}
By the formula \eqref{eq:Bohr}, one has 
\begin{align}\label{eq-bohr-chi}
\|f\|_{\cal H^\infty}= &\sup_{\chi \in \T^\infty}  \bigg| \sum_{1\leq i,j\leq d} a_{ij}  \chi(m_in_j)\bigg|= \sup_{\chi \in \T^\infty}  \bigg| \sum_{1\leq i,j\leq d} a_{ij}  \chi(m_i)\chi(n_j)\bigg|.
\end{align}
Observe that the latter supremum in \eqref{eq-bohr-chi} is bounded above by
\begin{align}\label{eq:operator-norm}
\sup_{|z_i|\leq 1,\, |w_j|\leq 1}
\bigg|\sum_{1\leq i,j\leq d} a_{ij} z_i w_j\bigg|.
\end{align}
By duality, \eqref{eq:operator-norm} is exactly the operator norm of \(A\), viewed as a linear map from \(\ell^\infty_d\) to \(\ell^1_d\) with respect to the standard basis. This completes the proof. 
\end{proof}

Clearly, if $f$ is of the form \eqref{eq-hyperbolic-type-f}, then
$$
\cal S_N f
=
\sum_{\substack{1\leq i,j\leq d\\ m_i n_j\leq N}}
a_{ij}(m_i n_j)^{-s}.
$$
The relevant hyperbolic structure is obtained by choosing
$$
m_i=i,\qquad n_j=\left\lfloor\frac{N}{j}\right\rfloor.
$$
We only need the following elementary number-theoretic property.

\begin{lemma}\label{prop-number-theoretic}
  Let $N\geq 10$ and set $d=\left\lfloor \sqrt{N}\right\rfloor-1$. Then, for integers $1\leq i, j \leq d$, we have $$i\left\lfloor \frac{N}{j} \right\rfloor \leq N \quad \text{if and only if}\quad i\leq j.$$
\end{lemma}
\begin{proof}
  If $i\leq j$, then $$i\left \lfloor \frac{N}{j}\right\rfloor \leq j\cdot \frac{N}{j} = N.$$ The sufficiency follows. For the necessity, we only need to show $$(j+1) \left\lfloor \frac{N}{j} \right\rfloor >N, \quad 1\leq j \leq d.$$ Note that $$\frac{N}{j}-\frac{N}{j+1} = \frac{N}{j(j+1)} > \frac{N}{d(d+1)}>1.$$ This implies $$\left\lfloor \frac{N}{j} \right\rfloor >\frac{N}{j+1}.$$ The proof is completed.
\end{proof}

Testing \(\cal S_N\) on the hyperbolic Dirichlet polynomial \(f_A\), and applying Lemma \ref{lem-matrix-dominated} together with Lemma \ref{prop-number-theoretic}, we obtain the following.

\begin{lemma}\label{lem-need-test-matrix}
  Let $N\geq 10$, and set $d=\left\lfloor \sqrt{N}\right\rfloor -1$. Then, for any matrix $A=(a_{ij})_{1\leq i,j\leq d}$, we have $$
\|\cal S_N\|_{\mathcal H^\infty\to \mathcal H^\infty}
\geq
\frac{\|f_{\widetilde A}\|_{\mathcal H^\infty}}
{\|A\|_{\ell^\infty_d\to \ell^1_d}},
$$ where $\widetilde{A}$ is the upper triangular truncation of $A$, that is, $\widetilde{A}=(\mathbf{1}_{i\leq j}a_{ij})_{1\leq i,j\leq d}.$
\end{lemma}

\begin{proof}
Clearly,
$$
\|\cal S_N\|_{\mathcal H^\infty\to \mathcal H^\infty}
\geq
\frac{\|\cal S_N f_A\|_{\mathcal H^\infty}}{\|f_A\|_{\mathcal H^\infty}}.$$ By Lemma \ref{prop-number-theoretic}, we see $$\cal S_N f_A=\sum_{\substack{1\leq i,j\leq d\\ i\left\lfloor \frac{N}{j} \right\rfloor\leq N}} a_{ij} \bigg( i \left\lfloor \frac{N}{j}\right\rfloor \bigg)^{-s} = \sum_{\substack{1\leq i,j\leq d\\ i\leq j}} a_{ij} \bigg( i \left\lfloor \frac{N}{j}\right\rfloor \bigg)^{-s} = f_{\widetilde{A}}.$$ By Lemma \ref{lem-matrix-dominated}, the denominator is bounded from above by $\|A\|_{\ell^\infty_d \to \ell^1_d}$. The proof is completed.
\end{proof}

\subsection{The matrix of the discrete Hilbert transform}\label{subsec-hilbert}
We now seek to maximize the latter ratio in Lemma \ref{lem-need-test-matrix} by suitably choosing the matrix \(A\). Motivated by the argument of Kwapie\'{n} and A. Pe{\l}czy{\'n}ski \cite{KP} for the main triangle projection, we consider the matrix \[
A
:=
\begin{pmatrix}
0 & 1 & \frac{1}{2}& \cdots & \frac{1}{d-1}\\
-1 & 0 & 1 &  \cdots & \frac{1}{d-2}\\
-\frac{1}{2} & -1 & 0 & \cdots & \frac{1}{d-3}\\
\vdots & \vdots & \vdots & \ddots & \vdots\\
-\frac{1}{d-1} & -\frac{1}{d-2} & -\frac{1}{d-3} & \cdots & 0
\end{pmatrix}.
\]
Namely, the entries of \(A\) are given by
\[
a_{ij}=
\begin{cases}
\dfrac{1}{j-i}, & i\neq j,\\
0, & i=j.
\end{cases}
\]
Note that $A$ is a $d\times d$-section of the discrete Hilbert transform
$$
H: \ell^2(\mathbb Z) \to \ell^2(\mathbb Z);\quad (x_i)_{i\in \mathbb Z}\mapsto \bigg(\sum_{\substack{j\in\mathbb Z\\ j\neq i}}
\frac{x_j}{j-i} \bigg)_{i\in \mathbb Z}.
$$ It follows from the Hilbert inequality (see, for example, \cite{Gra94}) that $$\|A\|_{\ell_d^2\to \ell_d^2} \leq \|H\|_{\ell^2(\mathbb Z)\to \ell^2(\mathbb Z)}=\pi.$$ In particular, for every \(x\in \ell^\infty_d\), one deduces that 
\[
\|Ax\|_{\ell^1_d}
\leq \sqrt d\,\|Ax\|_{\ell^2_d}
\leq \pi\sqrt d\,\|x\|_{\ell^2_d}
\leq \pi d\,\|x\|_{\ell^\infty_d}.
\]
This gives
\begin{align}\label{eq-upper-norm-A}
  \|A\|_{\ell^\infty_d\to \ell^1_d} \leq \pi d .
\end{align} 

On the other hand, $$\widetilde{A}=  \begin{pmatrix}
0 & 1 & \frac{1}{2}& \cdots & \frac{1}{d-1}\\
0 & 0 & 1 &  \cdots & \frac{1}{d-2}\\
0 & 0 & 0 & \cdots & \frac{1}{d-3}\\
\vdots & \vdots & \vdots & \ddots & \vdots\\
0 & 0 & 0 & \cdots & 0
\end{pmatrix},$$ and hence $$f_{\widetilde{A}} = \sum_{1\leq i<j\leq d} \frac{1}{j-i} \bigg( i \left\lfloor \frac{N}{j}\right\rfloor \bigg)^{-s}.$$ Note that the coefficients of $f_{\widetilde{A}}$ are all non-negative, so 
\begin{equation}
		\begin{split}
  \|f_{\widetilde{A}}\|_{\cal H^\infty} =& \sum_{1\leq i < j \leq d} \frac{1}{j-i}
  = \sum_{j=2}^d \bigg(1+\frac{1}{2}+\cdots + \frac{1}{j-1}\bigg)\\= & d\log d+(\gamma-1)d+O(1),
  \end{split}\label{eq-fA-tilde}
\end{equation}
where $\gamma=0.57721...$ is the Euler--Mascheroni constant.
Substituting \eqref{eq-upper-norm-A} and \eqref{eq-fA-tilde} into Lemma \ref{lem-need-test-matrix}, and using $d=\sqrt{N} +O(1)$, we obtain $$\|\cal S_N\|_{\cal H^\infty \to \cal H^\infty} \geq \frac{1}{2\pi}\log N+\frac{\gamma-1}{\pi}+ O(N^{-1/2}).$$ 
The proof is completed. 

\vspace{0.2cm}

\noindent \textbf{AI Disclosure.}
The proof of the logarithmic lower bound was developed with the assistance of OpenAI's ChatGPT-5.5 Plus. After the authors formulated the problem and interacted with ChatGPT over several rounds, ChatGPT identified a suitable test function and supplied a complete proof, which the authors subsequently verified. The original proof invoked the Baker--Harman--Pintz theorem on primes in short intervals \cite{BHP}; the authors later eliminated this dependence by introducing hyperbolic Dirichlet polynomials and replacing the relevant step with an elementary argument. The authors also observed that the matrix-theoretic idea in Section~\ref{subsec-hilbert}---testing the main triangle projection on a discrete Hilbert matrix---is essentially due to S. Kwapie{\'n} and A. Pe{\l}czy{\'n}ski \cite{KP}. This connection was not identified by ChatGPT. Since \cite{KP} does not record the explicit constant $1/\pi$ needed here, the short argument is retained. The final draft has been thoroughly rewritten and reorganized by the authors, who take full responsibility for the contents of this manuscript.

\vspace{0.2cm}

\noindent\textbf{Acknowledgements.}
The authors are grateful to Yanqi Qiu and Lai Jiang for their valuable discussions and helpful comments.

\end{document}